\pdfoutput=1
\documentclass[12pt,a4paper]{amsart}
\usepackage[a4paper, left=28mm, right=28mm, top=28mm, bottom=34mm]{geometry}

\usepackage{hyperref}
\hypersetup{nesting=true,debug=true,naturalnames=true}
\usepackage{graphicx,amssymb,upref}
\newcommand{\mrev}[1]{\href{http://www.ams.org/mathscinet-getitem?mr=#1}{MR#1}}
\newcommand{\zbl}[1]{\href{http://www.emis.de/cgi-bin/MATH-item?#1}{Zbl #1}}
\newcommand{\jfm}[1]{\href{http://www.emis.de/cgi-bin/MATH-item?#1}{JFM #1}}

\let\<\langle
\let\>\rangle
\newcommand{\doi}[1]{doi:\,\href{http://dx.doi.org/#1}{#1}}

\let\uml\"

\usepackage{bm}
\usepackage{mathtools}
\usepackage{longtable}

\newtheorem{lem}{Lemma}
\newtheorem{prop}[lem]{Proposition}
\newtheorem{thm}[lem]{Theorem}
\newtheorem{cor}[lem]{Corollary}
\newtheorem{claim}[lem]{Claim}

\numberwithin{lem}{section}

\theoremstyle{definition}
\newtheorem{rem}[lem]{Remark}

\DeclareMathOperator{\Per}{Per}
\DeclareMathOperator{\ord}{ord}

\newcommand{\C}{\mathbb{C}}
\newcommand{\F}{\mathbb{F}}
\newcommand{\Q}{\mathbb{Q}}
\newcommand{\Z}{\mathbb{Z}}

\newcommand{\Jac}{\mathrm{Jac}}

\newcommand{\seq}[1]{\bm{#1}}

\numberwithin{equation}{section}

\title[Periods of sequences associated with division polynomials]{Periods modulo $p$ of integer sequences associated with division polynomials of genus $2$ curves}

\author[Y.\ Ishitsuka]{Yasuhiro Ishitsuka}
\address{
Institute of Mathematics for Industry, Kyushu University,
Fukuoka, 819-0395, Japan}
\email{yishi1093@gmail.com}

\author[T.\ Ito]{Tetsushi Ito}
\address{Department of Mathematics, Faculty of Science, Kyoto University, Kyoto 606-8502, Japan}
\email{tetsushi@math.kyoto-u.ac.jp}

\author[T.\ Ohshita]{Tatsuya Ohshita}
\address{Department of Mathematics, Cooperative Faculty of Education, Gunma University,
Maebashi, Gunma 371-8510, Japan}
\email{ohshita@gunma-u.ac.jp}

\author[T.\ Taniguchi]{Takashi Taniguchi}
\address{Department of Mathematics, Graduate School of Science, Kobe University, Kobe 657-8501, Japan}
\email{tani@math.kobe-u.ac.jp}

\author[Y.\ Uchida]{Yukihiro Uchida}
\address{Department of Mathematical Sciences, Graduate School of Science, Tokyo Metropolitan University, 1-1 Minami-Osawa, Hachioji, Tokyo 192-0397, Japan}
\email{yuchida@tmu.ac.jp}


\begin{document}

\begin{abstract}
We study an integer sequence associated with Cantor's division polynomials of a genus 2 curve having an integral point.
We show that the reduction modulo $p$ of such a sequence is periodic for all but finitely many primes $p$, and describe the relation between the period of the reduction modulo $p$ of the sequence and the order of the integral point on the reduction modulo $p$ in the Jacobian variety explicitly.
This generalizes Ward's results on elliptic divisibility sequences associated with division polynomials of elliptic curves.
\end{abstract}

\maketitle
\setcounter{tocdepth}{1}
\tableofcontents

\section{Introduction}
An integer sequence $\{ a_n \}_{n\in\mathbb{Z}}$ is called a \emph{divisibility sequence} if $a_m \mid a_n$ whenever $m \mid n$.
An \emph{elliptic divisibility sequence} is a divisibility sequence $\seq{W} \coloneqq \{ W_n \}_{n\in\mathbb{Z}}$ satisfying
\[
	W_{n+m} W_{n-m}
	= W_{n+1} W_{n-1} W_m^2 - W_{m+1} W_{m-1} W_n^2
\]
for all integers $m, n \in \Z$.
Elliptic divisibility sequences were introduced by Ward~\cite{Ward}.
Ward proved that for an arbitrary ``non-degenerate'' elliptic divisibility sequence $\seq{W}$, there exist an elliptic curve $E$ defined over $\Q$ and $P=(x_P,y_P) \in E(\Q)$ such that $\psi_n(x_P,y_P)=W_n$,
where $\psi_n(X,Y) \in \Q[X,Y]$ is the $n$-th division polynomial of $E$.
Using them, he also proved that the reduction modulo $p$ of the sequence $\seq{W}$ is periodic for all but finitely many primes $p$.
More precisely, he proved the following:
Let $\Per_p(\seq{W})$ be the period of
the reduction modulo $p$ of the sequence $\seq{W}$.
Let $\ord_p({P})$ be the order of the point $\overline{P} \in E(\F_p)$,
where $\overline{P}$ is the reduction of $P$ modulo $p$.
Then
$\ord_p({P})$ divides $\Per_p(\seq{W})$,
and $\Per_p(\seq{W})$ divides $(p - 1) \ord_p({P})$, i.e.
\[
    \ord_p({P}) \mid \Per_p(\seq{W}) \mid (p - 1) \ord_p({P})
\]
(see \cite[Theorem 10.1]{Ward}).

The aim of this paper is to generalize these results to genus $2$ curves with integral points.
In order to state our results,
let us introduce some notation.
Let $C$ be a hyperelliptic curve of genus $2$ over $\Q$ defined by
\[
	Y^2 = F(X) \coloneqq X^5 + a_4 X^4 + a_3 X^3 + a_2 X^2 + a_1 X + a_0,
\]
where $a_0,a_1,a_2,a_3,a_4\in \Z$.
Let $\mathrm{disc}(F) \in \Z$ be the discriminant of $F(X)$,
and $\Jac(C)$ be the Jacobian variety of $C$.
For an integer $n \geq 0$, let $\psi_n(X) \in \Z[X]$ be
the division polynomial of $C$ defined by Cantor \cite{Cantor}.
Let $P = (x_P,y_P) \ (x_P, y_P \in \Z)$ be an integral point on $C \backslash \{ \infty \}$.
We put
\[
  D_P \coloneqq [P] - [\infty] \in \Jac(C)(\Q)
  \qquad \text{and} \qquad
  c_n \coloneqq \psi_n(x_P) \in \Z.
\]

The main results of this paper are as follows.

\begin{thm}
\label{MainTheorem}
Let $\seq{c} \coloneqq \{ c_n \}_{n\in\mathbb{Z}} \coloneqq \{ \psi_n(x_P) \}_{n\in\mathbb{Z}}$ be the integer sequence associated with the division polynomials of a hyperelliptic curve $C$ and its integral point $P$ on $C \backslash \{ \infty \}$ defined as above.
Assume that $c_3 c_4 c_5 c_6 c_7 (c_4^3 - c_3^3 c_5) \neq 0$.
Let $p$ be an odd prime
which divides neither
$\mathrm{disc}(F)$ nor $c_3 c_4 c_5 c_6 c_7 (c_4^3 - c_3^3 c_5)$.
Then the following assertions hold.
\begin{enumerate}
\item The reduction modulo $p$ of the sequence $\seq{c}$ is periodic.
\item Let $\Per_p(\seq{c})$ be the period of
the reduction modulo $p$ of the sequence $\seq{c}$.
Let $\overline{D_P} \in \Jac(C)(\F_p)$ be the reduction modulo $p$ of $D_P$,
and $\ord_p(D_P)$ be the order of the point $\overline{D_P} \in \Jac(C)(\F_p)$.
Then 
$\ord_p(D_P)$ divides $\Per_p(\seq{c})$,
and $\Per_p(\seq{c})$ divides $(p-1) \ord_p(D_P)$, i.e.
    \[
		\ord_p(D_P) \mid \Per_p(\seq{c}) \mid (p-1) \ord_p(D_P).
    \]
\end{enumerate}
\end{thm}

Since
$|\Jac(C)(\F_p)| \leq (1 + \sqrt{p})^4$
by the Hasse--Weil bound
(see \cite[Theorem 19.1, (b) and (c)]{Milne}),
we obtain the following upper bound of $\Per_p(\seq{c})$.

\begin{cor}
\label{Corollary:MainTheorem}
The period $\Per_p(\seq{c})$ of the reduction modulo $p$ of the sequence $\seq{c}$
is bounded above by $(p-1)(1 + \sqrt{p})^4$.
\end{cor}

Theorem \ref{MainTheorem} (2) means that the ratio $\Per_p(\seq{c})/\ord_p(D_P)$
is an integer and a divisor of $p-1$.
The method in this paper in fact allows us to give an
explicit description of this ratio,
which is an analogue of Ward's result for elliptic divisibility sequences \cite[Theorem~10.1]{Ward}.
As a precise version of Theorem \ref{MainTheorem} (2),
we prove the following.
\begin{thm}
\label{MainTheorem2}
Under the assumptions in Theorem~\ref{MainTheorem},
let $r\coloneqq \ord_p(D_P)$ be the order of $\overline{D_P} \in \Jac(C)(\F_p)$, and $\alpha_p, \beta_p \in \F_p$ be elements satisfying
$\alpha_p \equiv c_{r+3}/(c_3 c_{r+2}) \pmod{p}$ and
$\beta_p \equiv (c_3^2 c_{r+2}^3)/c_{r+3}^2 \pmod{p}$,
where we know $c_{r+2}, c_{r+3} \not\equiv 0 \pmod{p}$ (see Claim \ref{claim:r+2,r+3}).
Let $d$ be the least positive integer such that
$\alpha_p^d \equiv \beta_p^{d^2} \equiv 1 \pmod{p}$.
Then we have
\[ \Per_p(\seq{c}) = d \ord_p(D_P). \]
\end{thm}

For a given sequence $\seq c$, the behavior of
$d=\ord_p(D_P)/\Per_p(\seq{c})$
as a divisor of $p-1$, in varying $p$, does
not seem to have an obvious pattern. It might
thus be interesting to seek the behavior
from, e.g., a statistical point of view
(see Remark \ref{rem:d-statistic}).

\begin{rem}
The order $r = \ord_p(D_P)$
 can be calculated as the least positive integer $r$ such that
$c_{r-1} \equiv c_r \equiv c_{r+1} \equiv 0 \pmod{p}$
(see Theorem~\ref{thm:Theta} (2)).
\end{rem}

\begin{rem}\label{rem:piegonhole}
The condition
$c_3 c_4 c_5 c_6 c_7 (c_4^3 - c_3^3 c_5) \neq 0$
in Theorem \ref{MainTheorem} seems technical.
We need to assume it in order to prove properties of
the reduction modulo $p$ of the sequence $\seq{c}$ by induction
(see the proof of Theorem \ref{thm:period_psi3}).
In fact, under a weaker assumption,
we can prove the periodicity of the reduction modulo $p$ of the sequence $\seq{c}$
by the pigeonhole principle.
We demonstrate it in Proposition \ref{prop:pigeonhole}.
Meanwhile, the upper bound of $\Per_p(\seq{c})$ obtained by the pigeonhole principle
is $p^{11}$, which is (much) larger than the upper bound obtained in Corollary \ref{Corollary:MainTheorem}.
\end{rem}

Although Theorem~\ref{MainTheorem} and 
Theorem~\ref{MainTheorem2}
are analogous to Ward's results for elliptic divisibility sequences, the proofs are quite different.
Ward's proof does not seem applicable to our case.
Our proofs of
Theorem~\ref{MainTheorem} and 
Theorem~\ref{MainTheorem2}
are similar to the proofs for elliptic divisibility sequences given by Shipsey and Swart \cite{SS}.
They used recurrence relations to prove Ward's results.
For genus $2$ curves,
Cantor proved that
$\seq{c}$ satisfies a bilinear recurrence relation of Somos 8 type \cite[p.143]{Cantor},
where a recurrence relation is said to be of Somos $k$ type
if it is of the form
\[
	c_n c_{n+k} = \sum_{i=1}^{\lfloor k/2 \rfloor}
	\alpha_i c_{n+i} c_{n+k-i}.
\]
However, the recurrence relation of Somos 8 type alone does not seem to imply
Theorem \ref{MainTheorem} and
Theorem \ref{MainTheorem2}.

In this paper, we shall first show that 
$\seq{c}$ satisfies the following recurrence relations for all integers $m$ and $n$
(see Theorem~\ref{thm:psi_Somos}):
	\begin{align*}
	\begin{split}
		c_4 c_{n+m} c_{n-m}
		&= c_{m+1} c_{m-1} c_{n+3} c_{n-3} \\
		& \quad + (c_4 c_m^2 - c_3^2 c_{m+1} c_{m-1})
		c_{n+2} c_{n-2} \\
		& \quad + (c_3^2 c_{m+2} c_{m-2} - c_{m+3} c_{m-3})
		c_{n+1} c_{n-1} \\
		& \quad - c_4 c_{m+2} c_{m-2} c_n^2,
	\end{split}\\
	\begin{split}
 		c_3 c_5 c_{n+m+1} c_{n-m}
		&= c_3 c_{m+2} c_{m-1} c_{n+4} c_{n-3} \\
		& \quad + (c_5 c_{m+1} c_m - c_3 c_4 c_{m+2} c_{m-1})
		c_{n+3} c_{n-2} \\
		& \quad + (c_3 c_4 c_{m+3} c_{m-2} - c_3 c_{m+4} c_{m-3})
		c_{n+2} c_{n-1} \\
		& \quad - c_5 c_{m+3} c_{m-2} c_{n+1} c_n.
	\end{split}
	\end{align*}
In fact, these recurrence relations are satisfied by Cantor's division polynomials $\{ \psi_n(X) \}_{n\in\mathbb{Z}}$, which may be of independent interest.
Specializing to $m=4$ and $5$, we obtain
bilinear recurrence relations of Somos $8$, $9$, $10$ and $11$ type
satisfied by $\seq{c}$ (see Corollary~\ref{cor:Somos8-11}),
which includes Cantor's recurrence relation mentioned above.
Using these as key ingredients, we prove Theorem \ref{MainTheorem} and Theorem \ref{MainTheorem2} by inductive arguments.

Note that some other sequences satisfying relations of
Somos type have appeared in the literature.
As examples of recent results, Hone~\cite{Hone} proved that
certain Hankel determinants corresponding to a genus $2$ curve
satisfy a relation of Somos $8$ type.
Doliwa~\cite{Doliwa} proved some bilinear relations
for multipole orthogonal polynomials via their determinantal expressions.

Independently of our work,
Ustinov~\cite[Theorem~1]{Ustinov} recently proved that
the reduction modulo an arbitrary integer of a sequence
satisfying a relation of Somos type are eventually periodic
if the sequence has finite rank.
Here a sequence $\{ s_n \}_{n\in\mathbb{Z}}$ has \emph{finite rank}
if the matrices
\[
	M_s^{(0)} = (s_{m+n} s_{m-n})_{m,n\in\mathbb{Z}}, \quad
	M_s^{(1)} = (s_{m+n+1} s_{m-n})_{m,n\in\mathbb{Z}}
\]
have finite rank.
This result is proved by several recurrence relations
of Somos type and the pigeonhole principle
similarly to Proposition~\ref{prop:pigeonhole}.
Ustinov's theorem can be applied to the case a modulus is not prime.
On the other hand, the upper bound of the period,
although it is not given explicitly in \cite{Ustinov},
is larger than our bound as discussed in Remark~\ref{rem:piegonhole}.

The outline of this paper is as follows.
In Section 2, we recall Cantor's division polynomials
of a genus $2$ curve and their basic properties.
Cantor's division polynomials are described
by the hyperelliptic sigma function.
A classical formula of theta functions proved by Caspary and
Frobenius shows that the sequence $\seq{c}$ satisfies some recurrence relations.
In Section 3, using the recurrence relation obtained in Section 2,
we prove the periodicity of the reduction modulo $p$ 
of the sequence $\seq{c}$.
In Section 4, we prove Theorem \ref{MainTheorem}
and Theorem \ref{MainTheorem2}.
In Appendix A, we prove a formula relating Cantor's division polynomials
and hyperelliptic sigma functions.
In Appendix B, we give a numerical example.
For the integer sequence introduced by Cantor (OEIS A058231),
we give numerical results on
the period of the reduction modulo $p$ of the sequence $\seq{c}$
and the order of a point on the reduction modulo $p$ of the Jacobian variety.

\section{Cantor's division polynomials}
In this section, we prove some properties of Cantor's division polynomials
used in the proof of Theorem~\ref{MainTheorem}.

Let $K$ be a field of characteristic different from $2$.
Let $C$ be a hyperelliptic curve of genus $2$ defined by
\[
	Y^2 = F(X) \coloneqq X^5 + a_4 X^4 + a_3 X^3 + a_2 X^2 + a_1 X + a_0,
\]
where $a_0,a_1,a_2,a_3,a_4\in K$.
Let $\Jac(C)$ be the Jacobian variety of $C$.
Let $\infty \in C$ be the point at infinity of $C$.
We embed $C$ into $\Jac(C)$ by $P \mapsto D_P \coloneqq [P] - [\infty]$.
The image of $C$ is written as $\Theta$, which is called the \textit{theta divisor} on $\Jac(C)$.

For an integer $n \geq 0$,
let $\psi_n(X) \in K[X]$ be the division polynomials of $C$ defined by Cantor; see \cite{Cantor} for details.
We extend the division polynomials for $n < 0$ by
$\psi_n(X) \coloneqq -\psi_{-n}(X)$.
For $-1\leq n\leq 3$, they are given by
\[
    \psi_{-1}(X) = \psi_0(X) = \psi_1(X) = 0, \quad
    \psi_2(X) = 1, \quad
    \psi_3(X) = 4 F(X).
\]

\begin{thm}
\label{thm:Theta}
Let $P = (x_P, y_P) \in C(K)$ be a $K$-rational point with $y_P \neq 0$, and $n\ge 3$.
The following assertions hold.
\begin{enumerate}
    \item $n D_P \in \Theta$ if and only if $\psi_n(x_P) = 0$.
    \item $n D_P = 0$ if and only if $\psi_{n-1}(x_P) = \psi_n(x_P) = \psi_{n+1}(x_P) = 0$.
\end{enumerate}
\end{thm}

\begin{proof}
See \cite[pp.~140--141]{Cantor}.
\end{proof}

\begin{lem}
\label{lem:consecutive4values}
Let $P = (x_P, y_P) \in C(K)$ be a point with $y_P \neq 0$.
For every integer $n \in \Z$, at least one of
\[
\psi_{n}(x_P),\ \psi_{n+1}(x_P),\ \psi_{n+2}(x_P),\ \psi_{n+3}(x_P)
\]
is not zero.
\end{lem}

\begin{proof}
Since $\psi_{-n}(X) = -\psi_{n}(X)$, $\psi_2(X) = 1 \neq 0$, and
$\psi_{-2}(X) = -1 \neq 0$,
we may assume $n \geq 3$.
By \cite[Lemma 3.29]{Cantor}, at least one of
$f_n$, $f_{n+1}$, $f_{n+2}$, $f_{n+3}$ is not zero,
where $f_r$ is a rational function on $C$ defined in \cite[Section 3, Section 8]{Cantor}.
We have $\psi_r(X) = (2 Y)^{(r^2 - r - 2)/2} f_r$; see \cite[p.133, (8.7)]{Cantor}.
Since $y_P \neq 0$, at least one of
$\psi_n(x_P)$, $\psi_{n+1}(x_P)$, $\psi_{n+2}(x_P)$, $\psi_{n+3}(x_P)$
is not zero.
\end{proof}

In the rest of this section, let $K$ be a subfield of $\C$.
Cantor's division polynomials $\psi_n(X)$ can
be expressed by using the hyperelliptic sigma function.
Let $\sigma\colon\C^2\to\C$ be the hyperelliptic sigma function associated with $C$.
(For recent developments on the theory of sigma functions, see \cite{BEL} and references therein.
We adopt the notation used in \cite{Onishi2002,Onishi}.)
We define
\[ \sigma_2(u) \coloneqq \frac{\partial\sigma(u)}{\partial u_2}, \]
where $u=(u_1,u_2)\in\C^2$.

The following theorem essentially follows from the description of Cantor's division polynomials in \cite[Appendix~A]{Onishi}
(see also \cite[p.~518]{Matsutani}),
but there are sign errors in the literature.
For the convenience of the readers,
we correct a proof in Appendix \ref{sec:analytic_psi}.

\begin{thm}
\label{thm:analytic_psi}
Let $P = (x_P, y_P)\in C(\C)$ be a point and
let $u\in\C^2$ be the point corresponding to $P$ (for the definition of $u$, see Lemma \ref{lem:expansion2}).
Then we have
\[
	2 y_P \psi_n(x_P) = (-1)^n \frac{\sigma(n 
	u)}{\sigma_2(u)^{n^2}}.
\]
\end{thm}

The following argument is almost the same as that in \cite[Section~6]{Uchida}.
\begin{prop}\label{prop:CF}
Let $d \ge 6$ be an even integer and $u^{(1)},u^{(2)},\dotsc,u^{(d)}\in\C^2$.
Then we have
\begin{equation}\label{eq:CF}
	\operatorname{pf} \begin{pmatrix}
		\sigma(u^{(i)} + u^{(j)}) \sigma(u^{(i)} - u^{(j)})
	\end{pmatrix}_{1\le i,j\le d} = 0,
\end{equation}
where $\operatorname{pf} A$ is the Pfaffian of $A$.
\end{prop}
\begin{proof}
See \cite[Corollary 6.2]{Uchida} or \cite[p.\ 473, Ex.\ v]{Baker}.
The proposition follows from similar formulas
for theta functions proved by
Caspary~\cite{Caspary} and Frobenius~\cite{Frobenius}.
\end{proof}

Let $P = (x_P, y_P)\in C(\C)$ be a point and
we put $c_n \coloneqq \psi_n(x_P)$.

\begin{thm}
\label{thm:psi_Somos}
For all integers $m$ and $n$, we have
	\begin{align}
	\label{eq:Somos_even}
	\begin{split}
		c_4 c_{n+m} c_{n-m}
		&= c_{m+1} c_{m-1} c_{n+3} c_{n-3} \\
		& \quad + (c_4 c_m^2 - c_3^2 c_{m+1} c_{m-1})
		c_{n+2} c_{n-2} \\
		& \quad + (c_3^2 c_{m+2} c_{m-2} - c_{m+3} c_{m-3})
		c_{n+1} c_{n-1} \\
		& \quad - c_4 c_{m+2} c_{m-2} c_n^2,
	\end{split}\\
	\label{eq:Somos_odd}
	\begin{split}
 		c_3 c_5 c_{n+m+1} c_{n-m}
		&= c_3 c_{m+2} c_{m-1} c_{n+4} c_{n-3} \\
		& \quad + (c_5 c_{m+1} c_m - c_3 c_4 c_{m+2} c_{m-1})
		c_{n+3} c_{n-2} \\
		& \quad + (c_3 c_4 c_{m+3} c_{m-2} - c_3 c_{m+4} c_{m-3})
		c_{n+2} c_{n-1} \\
		& \quad - c_5 c_{m+3} c_{m-2} c_{n+1} c_n.
	\end{split}
	\end{align}
\end{thm}
\begin{proof}
Setting $d=6$, $u^{(1)}=nu$, $u^{(2)}=mu$,
$u^{(3)}=3u$, $u^{(4)}=2u$, $u^{(5)}=u$ and $u^{(6)}=0$ in \eqref{eq:CF},
we obtain \eqref{eq:Somos_even} by Theorem~\ref{thm:analytic_psi}
and Proposition \ref{prop:CF}.
Similarly, setting $u^{(1)}=(n+1/2)u$, $u^{(2)}=(m+1/2)u$,
$u^{(3)}=7u/2$, $u^{(4)}=5u/2$, $u^{(5)}=3u/2$ and $u^{(6)}=u/2$ in \eqref{eq:CF},
we obtain \eqref{eq:Somos_odd} by Theorem~\ref{thm:analytic_psi} and Proposition \ref{prop:CF}.
Note that we used $c_0 = c_1 = 0$ and $c_2 = 1$.
\end{proof}

By letting $m=4$ and $5$ in each of the above,
we obtain bilinear recurrence relations of Somos $8$, $9$, $10$ and $11$ type
satisfied by $\seq{c}$.

\begin{cor}
\label{cor:Somos8-11}
\begin{gather}
\label{eq:psi_Somos8}
	\begin{split}
	c_4 c_{n+4} c_{n-4}
	&= c_3 c_5 c_{n+3} c_{n-3}
	+ (c_4^3 - c_3^3 c_5) c_{n+2} c_{n-2}\\
	&\quad+ c_3^2 c_6 c_{n+1} c_{n-1}
	- c_4 c_6 c_n^2,
	\end{split}\\
\label{eq:psi_Somos9}
	\begin{split}
			c_3 c_5 c_{n+5} c_{n-4}
			&= c_3^2 c_6 c_{n+4} c_{n-3}
			+ c_4 (c_5^2 - c_3^2 c_6)
			c_{n+3} c_{n-2} \\
			&\quad + c_3 c_4 c_7 c_{n+2} c_{n-1}
			- c_5 c_7 c_{n+1} c_n,
	\end{split}\\
\label{eq:psi_Somos10}
	\begin{split}
		c_4 c_{n+5} c_{n-5}
		&= c_4 c_6 c_{n+3} c_{n-3}
		+ c_4 (c_5^2 - c_3^2 c_6) c_{n+2} c_{n-2} \\
		&\quad + (c_3^3 c_7 - c_8) c_{n+1} c_{n-1}
		- c_3 c_4 c_7 c_n^2,
	\end{split}\\
\label{eq:psi_Somos11}
	\begin{split}
			c_3 c_5 c_{n+6} c_{n-5}
			&= c_3 c_4 c_7 c_{n+4} c_{n-3}
			+ (c_5^2 c_6 - c_3 c_4^2 c_7)
			c_{n+3} c_{n-2} \\
			&\quad + c_3 (c_3 c_4 c_8 - c_9) c_{n+2} c_{n-1}
			- c_3 c_5 c_8 c_{n+1} c_n.
	\end{split}
\end{gather}
\end{cor}

Note that the Somos 8 type relation \eqref{eq:psi_Somos8} was proved by Cantor~\cite[p.~143]{Cantor}.

\section{Periodicity of the values of Cantor's division polynomials}
In this section,
we prove the periodicity of the reduction modulo $p$ of the values of Cantor's division polynomials.
As in Section 1, let
$C$ be a hyperelliptic curve of genus $2$ over $\Q$ defined by
\[
	Y^2 = F(X) \coloneqq X^5 + a_4 X^4 + a_3 X^3 + a_2 X^2 + a_1 X + a_0,
\]
where $a_0,a_1,a_2,a_3,a_4\in \Z$.
For an integer $n \geq 0$, let $\psi_n(X) \in \Z[X]$ be
the division polynomial of $C$ defined by Cantor.
Let $P = (x_P,y_P) \ (x_P, y_P \in \Z)$ be an integral point on $C \backslash \{ \infty \}$.
We put
\[
D_P \coloneqq [P] - [\infty] \in \Jac(C)(\Q)
\quad \text{and} \quad
c_n \coloneqq \psi_n(x_P) \in \Z.
\]

\begin{thm}
\label{thm:period_psi3}
        Let $p$ be an odd prime which is not a divisor of
        the discriminant of $F(X)$.
        We also assume that $p$ is not a divisor of
$
	   c_3 c_4 c_5 c_6 c_7
	   (c_4^3 - c_3^3 c_5)
$.
        Let $\overline{D_P} \in \Jac(C)(\F_p)$ be the reduction modulo $p$ of $D_P$,
        and $r\coloneqq \ord_p(D_P)$ be the order of $\overline{D_P}$.
        Then we have the following:
        \begin{enumerate}
            \item We have $c_{r+2}, c_{r+3} \not\equiv 0 \pmod{p}$.
            \item Let $\alpha_p, \beta_p \in \F_p$ be elements satisfying
            \[
                \alpha_p \equiv c_{r+3}/(c_3 c_{r+2}) \pmod{p}, \quad \beta_p \equiv (c_3^2 c_{r+2}^3)/c_{r+3}^2 \pmod{p}.
            \]
            Then, we have the following relations
            for all integers $n$ and $k$:
            \begin{equation}\label{eq:psi_modp_identiy}
            c_{kr+n} \equiv \alpha_p^{kn} \beta_p^{k^2} c_n \pmod{p}.
            \end{equation}
            \item We have $\alpha_p^r = \beta_p^2$ in $\F_p$.
        \end{enumerate}
\end{thm}

Note that the conditions in Theorem~\ref{thm:period_psi3} are satisfied for all but finitely many $p$.

The proof of Theorem \ref{thm:period_psi3} is divided into several steps.
In principle, the strategy of our proof is similar to the proof for elliptic divisibility sequences by Shipsey and Swart~\cite[Theorem~2]{SS}.
However, our proof is more involved than theirs.
We need to analyze the reduction modulo $p$ of the sequence using recurrence relations of Somos $8$, $9$, $10$, $11$ type together.

In order to simplify the notation,
we omit ``$(\text{mod} \ p)$'' in the rest of this section.
All the congruences are taken modulo $p$.

\begin{claim}
$y_P \not\equiv 0$.
\end{claim}

\begin{proof}
Since $c_3 = \psi_3(x_P) = 4 F(x_P)$ and $c_3 \not\equiv 0$,
we have $F(x_P) \not\equiv 0$.
This implies $y_P \not\equiv 0$.
\end{proof}

\begin{claim}
\label{claim:rgeq9}
The order $r = \ord_p(D_P)$ satisfies $r \ge 9$.
\end{claim}

\begin{proof}
Note that $\overline{D_P} \neq 0 \in \Jac(C)(\F_p)$
since $x_P, y_P \in \Z$.
Since $y_P \not\equiv 0$, we have $r \geq 3$.
By Theorem~\ref{thm:Theta} (2) with $n=r$,
we have $c_{r-1} \equiv c_{r} \equiv c_{r+1} \equiv 0$. 
Since $c_3 c_4 c_5 c_6 c_7 \not\equiv 0$
by our assumption,
we have $r \geq 9$.
\end{proof}

\begin{claim}
\label{claim:r+2,r+3}
$c_{r+2}, c_{r+3} \not\equiv 0$. \end{claim}

\begin{proof}
Since $c_{r-1} \equiv c_{r} \equiv c_{r+1} \equiv 0$, by Lemma \ref{lem:consecutive4values}, we have $c_{r+2} \not\equiv 0$. By our assumption, $c_3 \not \equiv 0$. By Theorem~\ref{thm:Theta} (1) with $n = 3$, we have $3 \overline{D_P} \notin \Theta$.
Since $r \overline{D_P} = 0$, we have $(r+3)\overline{D_P} \notin \Theta$.
Therefore, again by Theorem~\ref{thm:Theta} (1) with $n=r+3$, we have $c_{r+3} \not \equiv 0$.
\end{proof}

This finishes the proof of the first assertion,
and it allows us to define $\alpha_p, \beta_p \in\F_p^\times$ as above.
We continue the proof of Theorem \ref{thm:period_psi3}.
As the base case of the induction,
we first prove \eqref{eq:psi_modp_identiy}
for $k=1$ and $-3\le n\le 7$:

\begin{claim}
\label{claim:period_psi3:2-1-base}
For integers $n$ satisfying $-3\le n\le 7$, we have
\begin{equation}
\label{eq:period_psi4}
	c_{r+n} \equiv \alpha_p^n \beta_p c_n.
\end{equation}
\end{claim}

\begin{proof}
Since
$c_{r-1} \equiv c_r \equiv c_{r+1} \equiv 0$,
\eqref{eq:period_psi4} holds for $n=-1,0,1$.
Meanwhile,
\eqref{eq:period_psi4} holds for $n=2,3$
by the definitions of $\alpha_p$ and $\beta_p$.

Setting $n=r+3$ in \eqref{eq:psi_Somos8}, we obtain
\[
	0 \equiv c_3^2 c_6 c_{r+4} c_{r+2}
	- c_4 c_6 c_{r+3}^2
\]
since $c_{r-1} \equiv c_r \equiv c_{r+1} \equiv 0$.
By the assumption of Theorem \ref{thm:period_psi3}, we have $c_3 c_6 \not\equiv 0$.
Since \eqref{eq:period_psi4} holds for $n=2,3$ and $c_2 = 1$, we obtain
\[
	c_{r+4} \equiv \frac{c_4 c_{r+3}^2}{c_3^2 c_{r+2}}
	\equiv \frac{c_4 (\alpha_p^3 \beta_p c_3)^2}{c_3^2 \cdot \alpha_p^2 \beta_p c_2}
	\equiv \alpha_p^4 \beta_p c_4.
\]
Hence \eqref{eq:period_psi4} holds for $n=4$.

Setting $n=r+3$ in \eqref{eq:psi_Somos9}, we obtain
\[
	0 \equiv c_3 c_4 c_7 c_{r+5} c_{r+2}
	- c_5 c_7 c_{r+4} c_{r+3}.
\]
By assumption, we have $c_3 c_4 c_7 \not\equiv 0$.
Since \eqref{eq:period_psi4} holds for $n=2,3,4$ and $c_2 = 1$,
we obtain
\[
	c_{r+5} \equiv
	\frac{c_5 c_{r+4} c_{r+3}}{c_3 c_4 c_{r+2}}
	\equiv \frac{c_5 \cdot \alpha_p^4 \beta_p c_4 \cdot \alpha_p^3 \beta_p c_3}
	{c_3 c_4 \cdot \alpha_p^2 \beta_p c_2}
	\equiv \alpha_p^5 \beta_p c_5.
\]
Hence \eqref{eq:period_psi4} holds for $n=5$.

Setting $n=r+4$ in \eqref{eq:psi_Somos8}, we obtain
\[
	0 \equiv (c_4^3 - c_3^3 c_5) c_{r+6} c_{r+2}
	+ c_3^2 c_6 c_{r+5} c_{r+3}
	- c_4 c_6 c_{r+4}^2.
\]
By the assumption of Theorem \ref{thm:period_psi3},
we have $c_4^3 - c_3^3 c_5 \not\equiv 0$.
Since \eqref{eq:period_psi4} holds for $n=2,3,4,5$ and $c_2 = 1$,
we obtain
\[
	c_{r+6}
	\equiv \frac{-c_3^2 c_6 c_{r+5} c_{r+3}
	+ c_4 c_6 c_{r+4}^2}{(c_4^3 - c_3^3 c_5) c_{r+2}}
	\equiv \frac{-\alpha_p^8 \beta_p^2 c_3^3 c_5 c_6
	+ \alpha_p^8 \beta_p^2 c_4^3 c_6}{(c_4^3 - c_3^3 c_5) \alpha_p^2 \beta_p c_2}
	\equiv \alpha_p^6 \beta_p c_6.
\]
Hence \eqref{eq:period_psi4} holds for $n=6$.

Setting $n=r+2$ in \eqref{eq:psi_Somos8}, we obtain
\[
	c_4 c_{r+6} c_{r-2}
	\equiv - c_4 c_6 c_{r+2}^2.
\]
By the assumption of Theorem \ref{thm:period_psi3},
we have $c_4 c_6 \not\equiv 0$.
Since $c_{-2} = -c_2 = -1$ and \eqref{eq:period_psi4} holds for $n=2,6$,
we obtain
\[
	c_{r-2} \equiv -\frac{c_6 c_{r+2}^2}{c_{r+6}}
	\equiv -\frac{\alpha_p^4 \beta_p^2 c_2^2 c_6}{\alpha_p^6 \beta_p c_6}
	\equiv \alpha_p^{-2} \beta_p c_{-2}.
\]
Hence \eqref{eq:period_psi4} holds for $n=-2$.

Setting $n=r+2$ in \eqref{eq:psi_Somos9}, we obtain
\[
	c_3 c_5 c_{r+7} c_{r-2}
	\equiv - c_5 c_7 c_{r+3} c_{r+2}.
\]
By the assumption of Theorem \ref{thm:period_psi3},
we have $c_3 c_5\not\equiv 0$.
Since $c_{-2} = -c_2$ and \eqref{eq:period_psi4} holds for $n=-2,2,3$,
\[
	c_{r+7} \equiv -\frac{c_7 c_{r+3} c_{r+2}}{c_3 c_{r-2}}
	\equiv -\frac{\alpha_p^5 \beta_p^2 c_2 c_3 c_7}{\alpha_p^{-2} \beta_p c_3 c_{-2}}
	\equiv \alpha_p^7 \beta_p c_7.
\]
Hence \eqref{eq:period_psi4} holds for $n=7$.

Setting $n=r+1$ in \eqref{eq:psi_Somos9}, we obtain
\[
	c_3 c_5 c_{r+6} c_{r-3}
	\equiv c_3^2 c_6 c_{r+5} c_{r-2}.
\]
By assumption, we have $c_3 c_5 c_6 \not\equiv 0$.
Since $c_{-3} = -c_3$ and \eqref{eq:period_psi4} holds for $n=-2,5,6$,
we obtain
\[
	c_{r-3} \equiv
	\frac{c_3 c_6 c_{r+5} c_{r-2}}{c_5 c_{r+6}}
	\equiv \frac{\alpha_p^3 \beta_p^2 c_{-2} c_3 c_5 c_6}{\alpha_p^6 \beta_p c_5 c_6}
	\equiv \alpha_p^{-3} \beta_p c_{-3}.
\]
Hence \eqref{eq:period_psi4} holds for $n=-3$.

Summarizing the above, we see that
\eqref{eq:period_psi4} holds for $-3\le n\le 7$.
\end{proof}

Next, we shall prove \eqref{eq:psi_modp_identiy} for $k=1$
and for all $n$ by induction:

\begin{claim}
\label{claim:period_psi3:2-1}
For all integers $n \in \Z$, we have
\begin{equation}\label{eq:psi_modp_idntity_k1}
	c_{r+n} \equiv \alpha_p^n \beta_p c_n.    
\end{equation}

\end{claim}

\begin{proof}
Suppose that \eqref{eq:psi_modp_idntity_k1} holds for $m\le n\le m+10$ for some $m \ge -3$.
We shall prove that the assertion holds for $n=m+11$.
By Lemma \ref{lem:consecutive4values}, at least
one of $c_m$, $c_{m+1}$, $c_{m+2}$ or $c_{m+3}$
is not congruent to $0$ modulo $p$.
So it is enough to consider the following four cases:
\begin{itemize}
    \item $c_m\not\equiv 0$ 
    \item $c_{m+1}\not\equiv 0$
    \item $c_{m+2}\not\equiv 0$
    \item $c_{m+3}\not\equiv 0$
\end{itemize}

We first consider the case $c_m\not\equiv 0$.
From \eqref{eq:psi_Somos11} for $n = m+5$,
we have
\begin{align}
	\label{eq:psi_Somos11a}
	c_3 c_5 c_{m+11} c_m
	&= \sum_{i=0}^3 S_i c_{m+6+i} c_{m+5-i},
\end{align}
where
\[
S_0 \coloneqq - c_3 c_5 c_8, \quad
S_1 \coloneqq c_3 (c_3 c_4 c_8 - c_9), \quad
S_2 \coloneqq c_5^2 c_6 - c_3 c_4^2 c_7, \quad
S_3 \coloneqq c_3 c_4 c_7.
\]
Similarly, from \eqref{eq:psi_Somos11} for $n = r+m+5$,
we have
\begin{align}
 \label{eq:psi_Somos11b}
	c_3 c_5 c_{r+m+11} c_{r+m}
	&= \sum_{i=0}^3 S_i c_{r+m+6+i} c_{r+m+5-i}
\end{align}
where $S_0$, $S_1$, $S_2$, $S_3$ are the same constants as above.

By \eqref{eq:psi_Somos11a},
since $c_3 c_5 c_m \not\equiv 0$,
we have
\[
	c_{m+11}
	\equiv \frac{1}{c_3 c_5 c_m} \sum_{i=0}^3 S_i c_{m+6+i} c_{m+5-i}.
\]
On the other hand,
by the induction hypothesis,
we have $c_{r+n} \equiv \alpha_p^n \beta_p c_n$ for $m\le n \le m+10$.
Hence, by \eqref{eq:psi_Somos11b}, we obtain
\begin{align*}
	c_{r+m+11}
	&\equiv \frac{1}{c_3 c_5 c_{r+m}}
    \sum_{i=0}^3 S_i c_{r+m+6+i} c_{r+m+5-i} \\
	&\equiv \frac{1}{\alpha_p^m \beta_p c_3 c_5 c_m}
    \sum_{i=0}^3 S_i \cdot \alpha_p^{m+6+i} \beta_p c_{m+6+i} \cdot \alpha_p^{m+5-i} \beta_p c_{m+5-i} \\
	&\equiv \frac{1}{\alpha_p^m \beta_p c_3 c_5 c_m}
	\sum_{i=0}^3 S_i \alpha_p^{2m+11} \beta_p^2 \cdot c_{m+6+i} c_{m+5-i} \\
    &\equiv \frac{\alpha_p^{m+11} \beta_p}{c_3 c_5 c_m}
	\sum_{i=0}^3 S_i c_{m+6+i} c_{m+5-i}.
\end{align*}
Comparing two equations, we have
\[
	c_{r+m+11} \equiv \alpha_p^{m+11} \beta_p c_{m+11}  \pmod{p},
\]
and thus \eqref{eq:psi_modp_idntity_k1} is true for $n = m+11$.

The other cases are proved in a similar manner.
Note that when
$c_{m+1}\not\equiv 0$,
$c_{m+2}\not\equiv 0$,
$c_{m+3}\not\equiv 0$,
we shall use
\eqref{eq:psi_Somos10},
\eqref{eq:psi_Somos9},
\eqref{eq:psi_Somos8},
respectively.
By induction, \eqref{eq:psi_modp_idntity_k1}
holds for all $n\ge -3$.

The assertion for $n\le -4$ is proved by similar arguments.
Let $m \leq -4$ and assume that the assertion holds for every $n > m$.
By Lemma \ref{lem:consecutive4values}, at least
one of $c_{m+8}$, $c_{m+9}$, $c_{m+10}$ or $c_{m+11}$
is not congruent to $0$ modulo $p$.
So it is enough to consider the following four cases:
\begin{itemize}
    \item $c_{m+8}\not\equiv 0$ 
    \item $c_{m+9}\not\equiv 0$
    \item $c_{m+10}\not\equiv 0$
    \item $c_{m+11}\not\equiv 0$
\end{itemize}
When $c_{m+11}\not\equiv 0$,
we obtain
\begin{align*}
	c_m
	&= \frac{1}{c_3 c_5 c_{m+11}} \sum_{i=0}^3 S_i c_{m+6+i} c_{m+5-i}
\end{align*}
from \eqref{eq:psi_Somos11} for $n = m+5$.
Thus, we prove the assertion for $c_m$ from the assertions for $c_n$ for $n > m$.
Similarly, when
$c_{m+10}\not\equiv 0$,
$c_{m+9}\not\equiv 0$,
$c_{m+8}\not\equiv 0$,
we shall use
\eqref{eq:psi_Somos10},
\eqref{eq:psi_Somos9},
\eqref{eq:psi_Somos8},
respectively.
\end{proof}

Next, we shall prove part (3) of Theorem \ref{thm:period_psi3}.

\begin{claim}
\label{claim:period_psi3:3}
$\alpha_p^r = \beta_p^2 \in \F_p$.
\end{claim}

\begin{proof}
Setting $n=2$ and $n=-r-2$ in \eqref{eq:period_psi4}, we have
\[
	c_{r+2} \equiv \alpha_p^2 \beta_p c_2, \quad
	c_{-2} \equiv \alpha_p^{-r-2} \beta_p c_{-r-2}.
\]
Since $c_{-2} = -c_2 = -1$ and $c_{-r-2} = -c_{r+2}$, we have
$\alpha_p^r = \beta_p^2$ in $\F_p$.
\end{proof}

Finally, we prove \eqref{eq:psi_modp_identiy} for all integers $k \in \Z$.

\begin{claim}
For all integers $n$ and $k$, we have
\[ c_{kr+n} \equiv \alpha_p^{kn} \beta_p^{k^2} c_n. \]
\end{claim}

\begin{proof}
By Claim \ref{claim:period_psi3:2-1}, the assertion holds for $k = 1$.
We shall prove the assertion by induction on $k$.
Assume that the assertion holds for some $k$.
Then we have
\[
	c_{(k+1)r+n} = c_{kr+(r+n)}
	\equiv \alpha_p^{k(r+n)} \beta_p^{k^2} c_{r+n}.
\]
Since $\alpha_p^r = \beta_p^2 \in \F_p$ by Claim \ref{claim:period_psi3:3},
we have
\[
	\alpha_p^{k(r+n)} \beta_p^{k^2} c_{r+n}
 \equiv (\beta_p^2)^k \alpha_p^{kn} \beta_p^{k^2} c_{r+n}
  \equiv \alpha_p^{kn} \beta_p^{k^2 + 2k} c_{r+n}.
\]
By the assertion for $k=1$, we have
$c_{r+n} \equiv \alpha_p^n \beta_p c_n$.
Hence we have
\[
\alpha_p^{kn} \beta_p^{k^2 + 2k} c_{r+n}
\equiv \alpha_p^{kn} \beta_p^{k^2 + 2k} \cdot \alpha_p^n \beta_p c_n
\equiv \alpha_p^{(k+1)n} \beta_p^{(k+1)^2} c_r.
\]
The assertion is proved for $k+1$.
By induction, the assertion is proved for all $k \geq 1$.

Since we have
\[ c_{-kr+n}
\equiv - c_{kr-n}
\equiv - \alpha_p^{k \cdot (-n)} \beta_p^{k^2} c_{-n}
\equiv \alpha_p^{(-k) \cdot n} \beta_p^{(-k)^2} c_{n},
\]
the assertion for $k < 0$ follows.
\end{proof}

The proof of Theorem \ref{thm:period_psi3} is complete.

\section{Proof of the main theorems}

We are now ready to prove Theorem~\ref{MainTheorem} and Theorem~\ref{MainTheorem2}.

\begin{proof}[Proof of Theorem~\ref{MainTheorem}]
Let $p$ be a prime satisfying the assumption in Theorem~\ref{thm:period_psi3}.
Substituting $k=p-1$ in Theorem~\ref{thm:period_psi3} (2), we have
\[
	c_{(p-1)r+n} \equiv \alpha_p^{(p-1)n} \beta_p^{(p-1)^2} c_n
	\equiv c_n \pmod{p}
\]
for all integers $n \in \Z$.
Hence $\{ c_n \pmod{p} \}_{n\in\mathbb{Z}}$ is periodic, and the period $\Per_p(\seq{c})$ is a divisor of
$(p-1)r = (p-1) \ord_p(D_P)$.

Next, we shall prove that $r = \ord_p(D_P)$ divides $s \coloneqq \Per_p(\seq{c})$.
Since $c_{-1} = c_{1} = c_{1} = 0$ and $c_2 = 1$,
we have $s \geq 4$.
Recall that
$y_P \not\equiv 0 \pmod{p}$.
Since $s$ is the period of the reduction modulo $p$ of the sequence $\seq{c}$,
we have $c_{s+i} \equiv c_i \equiv 0 \pmod{p}$
for $i = -1, 0, 1$.
Therefore, by Theorem~\ref{thm:Theta} (2),
we obtain $s \overline{D_P} = 0$ in $\Jac(C)(\F_p)$.
Hence $r$ divides $s$.
\end{proof}

\begin{proof}[Proof of Theorem~\ref{MainTheorem2}]
Let $r \coloneqq \ord_p(D_P)$, $s \coloneqq \Per_p(\seq{c})$, and $k \coloneqq s/r$.
By Theorem~\ref{MainTheorem} (2), $k$ is a positive integer.
By Theorem~\ref{thm:period_psi3} (2), we have
$c_{dr+n} \equiv c_n \pmod{p}$ for all integers $n\in\Z$.
Hence we have $s = kr \mid dr$, which implies $k \mid d$.

Setting $n=2, 3$ in the relation in Theorem~\ref{thm:period_psi3} (2), we have
\[
	c_{kr+2} \equiv \alpha_p^{2k} \beta_p^{k^2} c_2 \pmod{p}, \quad
	c_{kr+3} \equiv \alpha_p^{3k} \beta_p^{k^2} c_3 \pmod{p}.
\]
Since $s=kr$ is the period and $c_2, c_3 \not\equiv 0 \pmod{p}$, we have
\[
	\alpha_p^k \equiv \beta_p^{k^2} \equiv 1 \pmod{p}.
\]
Hence we obtain $d \mid k$ since $d$ is the least positive integer satisfying such a condition (see \cite[Lemma~10.1]{Ward}).
Therefore, we have $d=k$, which implies $\Per_p(\seq{c}) = d \ord_p(D_P)$.
\end{proof}

As we mentioned in Remark \ref{rem:piegonhole},
we can prove Theorem~\ref{MainTheorem} (1) and a half of Theorem~\ref{MainTheorem} (2)
by using the pigeonhole principle instead of using Theorem~\ref{thm:period_psi3}:

\begin{prop}
\label{prop:pigeonhole}
Let $p$ be an odd prime which divides neither
$\mathrm{disc}(F)$ nor $c_3 c_4 c_5$.
Then the reduction modulo $p$ of the sequence $\seq{c}$ is periodic,
and we have
$\ord_p(D_P) \mid \Per_p(\seq{c})$.
\end{prop}

\begin{proof}
By Lemma \ref{lem:consecutive4values}, there exists no integer $m$ such that
\[
  c_m \equiv c_{m+1} \equiv c_{m+2} \equiv c_{m+3} \equiv 0\pmod{p}. \]
Since $c_3 c_4 c_5 \not\equiv 0 \pmod{p}$,
by the bilinear recurrence relations of Somos $8$, $9$, $10$ and $11$ type
in Corollary~\ref{cor:Somos8-11},
the values
$c_{m+11} \pmod p$ and $c_{m-1} \pmod p$ are uniquely determined by the values $c_{m+i} \pmod p$ for $0\le i\le 10$.
By the pigeonhole principle, there exist an integer $k \in \Z$ and a positive integer $s \geq 1$ such that $c_{s+k+i} \equiv c_{k+i} \pmod{p}$ for $0\le i\le 10$.
Thus, we obtain $c_{n+s} \equiv c_n \pmod{p}$ for all $n \in \Z$ by induction.

The proof of ``$\ord_p(D_P) \mid \Per_p(\seq{c})$''
is the same as Theorem~\ref{MainTheorem} (2).
(Note that the proof of ``$\ord_p(D_P) \mid \Per_p(\seq{c})$''
does not require Theorem~\ref{thm:period_psi3}.)
\end{proof}

\begin{rem}
In contrast to Theorem \ref{MainTheorem}, in the above proof of Proposition \ref{prop:pigeonhole},
we do not require the assumption that $c_6 c_7 (c_4^3 - c_3^3 c_5) \not\equiv 0 \pmod{p}$.
However, the upper bound for the period $\Per_p(\seq{c})$
we can obtain from the pigeonhole principle is $p^{11}$,
which is much larger than the upper bound in Corollary~\ref{Corollary:MainTheorem}.
In particular, without Theorem~\ref{thm:period_psi3},
it seems difficult to prove the divisibility
``$\Per_p(\seq{c}) \mid (p-1) \ord_p(D_P)$.''
\end{rem}

\appendix

\section{Proof of Theorem~\ref{thm:analytic_psi}}
\label{sec:analytic_psi}

In this appendix, we give a proof of Theorem~\ref{thm:analytic_psi}.
This result essentially follows from the description of Cantor's division polynomials
in \cite[Appendix]{Onishi}.
However, the sign in the formula in \cite[Theorem~A~1]{Onishi} is incorrect.
In fact, the sign $(-1)^{(2n-g)(g-1)/2}$ in \cite[Proposition~8.2 (ii)]{Onishi} should be replaced by $(-1)^{(n-g-1)(n+g^2+2g)/2}$ as in \cite[Theorem~5.1]{Uchida}.
Moreover, the sign $(-1)^{r(r-1)/2}$ in \cite[p.~738]{Onishi} should be read $(-1)^{(r-g)(r-g+1)/2}$.
Here we supply necessary arguments
to correct the sign errors in the literature.

For details on the hyperelliptic sigma function, we refer the readers to \cite{BEL} and references therein.
We adopt the definitions in \cite{Onishi2002,Onishi}.
In an expression for the Laurent expansion of a function, the symbol $(d^\circ(z_1,z_2,\ldots,z_m)\ge n)$ stands for the terms of total degree at least $n$ with respect to the variables $z_1,z_2,\ldots,z_m$.

We define differential forms
\[
	\omega_1 \coloneqq \frac{dX}{2Y}, \quad
	\omega_2 \coloneqq \frac{XdX}{2Y}, \quad
	\eta_1 \coloneqq \frac{(3X^3+2a_1X^2+a_2X)dX}{2Y}, \quad
	\eta_2 \coloneqq \frac{X^2dX}{2Y}.
\]
Let $\{ \alpha_1, \alpha_2, \beta_1, \beta_2 \}$ be a symplectic basis of $H_1(C(\C),\Z)$.
We define $2\times 2$ matrices by
\[
\omega' \coloneqq 
\begin{pmatrix}
\int_{\alpha_1} \omega_1 & \int_{\alpha_2} \omega_1 \\*[3mm]
\int_{\alpha_1} \omega_2 & \int_{\alpha_2} \omega_2
\end{pmatrix},
\qquad
\omega'' \coloneqq 
\begin{pmatrix}
\int_{\beta_1} \omega_1 & \int_{\beta_2} \omega_1 \\*[3mm]
\int_{\beta_1} \omega_2 & \int_{\beta_2} \omega_2
\end{pmatrix},
\]
\[
\eta' \coloneqq 
\begin{pmatrix}
\int_{\alpha_1} \eta_1 & \int_{\alpha_2} \eta_1 \\*[3mm]
\int_{\alpha_1} \eta_2 & \int_{\alpha_2} \eta_2
\end{pmatrix},
\qquad
\eta'' \coloneqq 
\begin{pmatrix}
\int_{\beta_1} \eta_1 & \int_{\beta_2} \eta_1 \\*[3mm]
\int_{\beta_1} \eta_2 & \int_{\beta_2} \eta_2
\end{pmatrix},
\]
which are called the \textit{period matrices}.

We define the \emph{hyperelliptic sigma function} by
\[
	\sigma(u) \coloneqq c \exp\left(-\frac{1}{2}
	\prescript{t}{}{u} \, \eta' \, \omega'^{-1} \, u \right)\vartheta
	\begin{bmatrix} \delta'' \\ \delta' \end{bmatrix}
	(\omega'^{-1} u,\ \omega'^{-1} \omega''),
\]
where $u = \begin{pmatrix} u_1 \\ u_2 \end{pmatrix} \in \C^2$,
$c$ is some constant, $\delta', \delta''$ are the Riemann constants, and $\vartheta$ is the Riemann theta function with characteristics.
The constant $c$ is determined so that the following lemma holds.
For details, see \cite[Lemma~1.2]{Onishi2002} and the references cited there.
\begin{lem}\label{lem:expansion1}
The function $\sigma(u)$ has the Taylor expansion
\[
	\sigma(u) = u_1 + \frac{1}{6} a_2 u_1^3 - \frac{1}{3} u_2^3 + (d^\circ(u_1,u_2)\ge 5)
\]
at $u = \begin{pmatrix} 0 \\ 0 \end{pmatrix}$.
\end{lem}

We also use the following lemmas.
\begin{lem}\label{lem:expansion2}
Let $P=(x_P,y_P)\in C(\C)$ and
\[
	u =
    \begin{pmatrix} u_1 \\ u_2 \end{pmatrix}
    =
    \begin{pmatrix}
	\int_\infty^{P} \omega_1 \\*[3mm] \int_\infty^{P} \omega_2
	\end{pmatrix}.
\]
Assume that $u$ is in a neighborhood of
$\begin{pmatrix} 0 \\ 0 \end{pmatrix}$.
Then we have
\begin{align}
	u_1 &= \frac{1}{3} u_2^3 + (d^\circ(u_2) \ge 4), \\
	\sigma_2(u) &= -u_2^2 + (d^\circ(u_2) \ge 3), \\
	x_P &= \frac{1}{u_2^2} + (d^\circ(u_2) \ge -1), \\
	y_P &= -\frac{1}{u_2^5} + (d^\circ(u_2) \ge -4).
\end{align}
\end{lem}
\begin{proof}
See \cite[Lemmas 1.7, 1.9, and 1.12]{Onishi2002}.
\end{proof}

\begin{lem}\label{lem:Psi}
The polynomial $\psi_n(X)\in\Z[X]$ is of degree $n^2-4$,
and its leading coefficient is $\binom{n+1}{3}$.
\end{lem}
\begin{proof}
The lemma follows from \cite[Theorem~8.17]{Cantor}.
\end{proof}

\begin{proof}[Proof of Theorem~\ref{thm:analytic_psi}]
Comparing the definition of $\psi_n(X)$ and the determinant expression of $\sigma(nu)/\sigma_2(u)^{n^2}$ in \cite[Theorem~A~1]{Onishi}, we have
\[
	2 y_P \psi_n(x_P) = \pm \frac{\sigma(nu)}{\sigma_2(u)^{n^2}}.
\]
To determine the sign, we compare the leading term
of the Laurent expansion of both sides at $u_2=0$.
By Lemmas~\ref{lem:expansion2} and~\ref{lem:Psi}, we have
\begin{equation}\label{eq:LHS}
	2 y_P \psi_n(x_P) = -2 \binom{n+1}{3} \frac{1}{u_2^{2n^2-3}} + (d^\circ(u_2) \ge -2n^2+4).
\end{equation}
By Lemmas~\ref{lem:expansion1} and~\ref{lem:expansion2}, we have
\begin{align*}
	\sigma(nu) &= n u_1 + \frac{1}{6} a_2 (n u_1)^3 - \frac{1}{3} (n u_2)^3 + (d^\circ(u_1,u_2)\ge 5) \\
	&= \frac{1}{3} n u_2^3 + \frac{1}{6} a_2 \left(\frac{1}{3} n u_2^3\right)^3 - \frac{1}{3} n^3 u_2^3 + (d^\circ(u_2)\ge 4) \\
	&= -2 \binom{n+1}{3} u_2^3 + (d^\circ(u_2)\ge 4).
\end{align*}
By Lemma~\ref{lem:expansion2}, we have
\[
	\sigma_2(u)^{n^2} = (-1)^{n^2} u_2^{2n^2} + (d^\circ(u_2)\ge 2n^2+1).
\]
Since $(-1)^{n^2}=(-1)^n$, we have
\begin{equation}\label{eq:RHS}
	\frac{\sigma(nu)}{\sigma_2(u)^{n^2}} = 2 (-1)^{n+1} \binom{n+1}{3} \frac{1}{u_2^{2n^2-3}} + (d^\circ(u_2) \ge -2n^2+4).
\end{equation}
Therefore, by \eqref{eq:LHS} and \eqref{eq:RHS}, we obtain
\[
	2 y_P \psi_n(x_P) = (-1)^n \frac{\sigma(nu)}{\sigma_2(u)^{n^2}}. \qedhere
\]
\end{proof}

\section{Numerical calculation of periods and orders}

Here we give an example illustrating 
Theorem \ref{MainTheorem}.
We study the integer sequence introduced by Cantor (see OEIS A058231)\footnote{\url{https://oeis.org/A058231}}.
It is an integer sequence $\{c_n\}_{n\geq0}$ satisfying
\begin{align*}
	& c_0 = c_1 = 0, \quad c_2 = 1, \quad c_3 = 36, \quad c_4 = -16, \notag \\
	& c_5 = 5041728, \quad c_6 = -19631351040, \quad c_7 = -62024429150208, \notag \\
	& c_8 = -2805793044443561984, \quad c_9 = -1213280369793911777918976
 \end{align*}
and the recurrence relation of Somos 8 type
\begin{multline*}
   {-16} c_{n} c_{n+8} - 181502208 c_{n+1} c_{n+7}
	+ 235226865664 c_{n+2} c_{n+6} \notag \\
   + 25442230947840 c_{n+3} c_{n+5} + 314101616640 c_{n+4}^2 = 0. \label{eq:variant}
\end{multline*}

It is a non-trivial fact that
such an integer sequence $\{ c_n \}_{n \geq 0}$
exists.
In fact, this sequence consists of values of
Cantor's division polynomials; see also \cite{Cantor}.
We set
\[
	C \colon Y^2 = X^5 - 3 X^4 - 2 X + 9, \quad P = (0,3).
\]
Let $\psi_n(X) \in \Z[X]$ be Cantor's division polynomial for $C$.
Then we can verify 
\[
	c_n = \psi_n(0).
\]
We extend the sequence $c_n$ to $n<0$ by $c_n = -c_{-n}$
(see OEIS A058231). In particular, we have
$c_{-1} = c_0 = c_1 = 0$.

From Theorem \ref{MainTheorem} and
Corollary \ref{Corollary:MainTheorem},
we obtain the following results.

\begin{cor}
Let $p$ be a prime not in the following list:
\[
    2, 3, 5, 7, 29, 41, 47, 379, 509, 853, 8059, 8753, 49711, 140891.
\]
Then the following assertions hold.
\begin{enumerate}
\item The reduction modulo $p$ of the sequence $\seq{c} = \{ c_n \}_{n\in\mathbb{Z}}$ is periodic.
\item Let $\Per_p(\seq{c})$ be the period of
the reduction modulo $p$ of the sequence $\seq{c}$.
Let $\ord_p(D_P)$ be the order of the point $\overline{D_P} \in \Jac(C)(\F_p)$.
Then we have
    \[
		\ord_p(D_P) \mid \Per_p(\seq{c}) \mid (p-1) \ord_p(D_P).
    \]
\item We have $\Per_p(\seq{c}) \leq (p-1)(1 + \sqrt{p})^4$.
\end{enumerate}
\end{cor}

\begin{proof}
By Theorem \ref{MainTheorem} and
Corollary \ref{Corollary:MainTheorem},
it is enough to determine the set of excluded primes.
The discriminant of $X^5 - 3 X^4 - 2 X + 9$ is $-36040475 = - 5^2 \times 29 \times 49711$.
(By Magma, the conductor of $C$ is $4613180800 = 2^7 \times 5^2 \times 29 \times 49711$.)
We calculate
    \begin{align*}
        c_3 &= 2^2 \times 3^2, \\
        c_4 &= -2^4, \\
        c_5 &= 2^6 \times 3^2 \times 8753,\\
        c_6 &= -2^8 \times 3 \times 5 \times 7 \times 41 \times 47 \times 379,\\
        c_7 &= - 2^{13} \times 3^2 \times 7 \times 853 \times 140891,\\
        c_4^3 - c_3^3 c_5 &=  - 2^{13} \times 7 \times 509 \times 8059.
    \end{align*}
\end{proof}

In the following table, for prime $p \leq 400$, we give numerical results on
the number of $\F_p$-rational points on the reduction modulo $p$
of $\Jac(C)$,
the order $\ord_p(D_P)$ of the point $\overline{D_P} \in \Jac(C)(\F_p)$,
the period $\Per_p(\seq{c})$ of the reduction modulo $p$ of the sequence $\seq{c}$,
the ratio $\Per_p(\seq{c})/\ord_p(D_P)$,
and the elements $\alpha_p, \beta_p \in \F_p$ in Theorem \ref{MainTheorem2}.

The calculations of $|\Jac(C)(\F_p)|$ and $\ord_p(D_P)$
are done by Magma \cite{Magma}.
The calculations of $\Per_p(\seq{c})$ are done by Sage \cite{Sage}
using the bilinear recurrence relations of Somos $8$, $9$, $10$ and $11$ type
satisfied by $\seq{c}$
in Corollary~\ref{cor:Somos8-11}.

\renewcommand{\arraystretch}{1.3}
\begin{longtable}{|r|c|c|c|c|c|c|}
\caption{Numerical verification of Theorem~\ref{MainTheorem} for the case of Cantor's sequence (OEIS A058231).}\\
\hline
$p$ & $|\Jac(C)(\F_p)|$ & $\ord_p(D_P)$ & $\Per_p(\seq{c})$ & $\Per_p(\seq{c})/\ord_p(D_P)$ & $\alpha_p$ & $\beta_p$ \\
\hline
2 & & & & & & \\
3 & 12 & 2 & 6 & 3 & & \\
5 & & & 12 & & & \\
7 & 28 & 7 & 21 & 3 & 4 & 2 \\
11 & 112 & 56 & 280 & 5 & 4 & 9 \\
13 & 127 & 127 & 762 & 6 & 10 & 7 \\
17 & 272 & 136 & 2176 & 16 & 10 & 4 \\
19 & 405 & 135 & 405 & 3 & 7 & 1 \\
23 & 692 & 173 & 3806 & 22 & 12 & 10 \\
29 & & & 2100 & & & \\
31 & 997 & 997 & 997 & 1 & 1 & 1 \\
37 & 1684 & 842 & 3368 & 4 & 6 & 31 \\
41 & 1693 & 1693 & 8465 & 5 & 10 & 37 \\
43 & 1186 & 1186 & 2372 & 2 & 42 & 1 \\
47 & 2433 & 2433 & 55959 & 23 & 18 & 17 \\
53 & 3284 & 821 & 10673 & 13 & 16 & 16 \\
59 & 3512 & 439 & 12731 & 29 & 45 & 19 \\
61 & 3910 & 3910 & 234600 & 60 & 26 & 40 \\
67 & 5056 & 632 & 41712 & 66 & 6 & 2 \\
71 & 5064 & 2532 & 88620 & 35 & 10 & 36 \\
73 & 5840 & 730 & 13140 & 18 & 37 & 57 \\
79 & 5825 & 5825 & 75725 & 13 & 18 & 52 \\
83 & 7324 & 3662 & 150142 & 41 & 78 & 77 \\
89 & 6762 & 2254 & 198352 & 88 & 60 & 75 \\
97 & 9884 & 9884 & 948864 & 96 & 90 & 2 \\
101 & 9900 & 275 & 13750 & 50 & 82 & 10 \\
103 & 10112 & 5056 & 10112 & 2 & 102 & 1 \\
107 & 12944 & 3236 & 343016 & 106 & 46 & 81 \\
109 & 11349 & 11349 & 306423 & 27 & 3 & 45 \\
113 & 12332 & 12332 & 1381184 & 112 & 12 & 41 \\
127 & 15272 & 15272 & 30544 & 2 & 126 & 1 \\
131 & 18724 & 9362 & 243412 & 26 & 45 & 86 \\
137 & 19104 & 9552 & 1299072 & 136 & 21 & 15 \\
139 & 20687 & 20687 & 2854806 & 138 & 71 & 72 \\
149 & 20696 & 5174 & 382876 & 74 & 37 & 64 \\
151 & 22010 & 22010 & 3301500 & 150 & 51 & 2 \\
157 & 27456 & 2288 & 118976 & 52 & 29 & 156 \\
163 & 26138 & 26138 & 4234356 & 162 & 137 & 122 \\
167 & 30036 & 7509 & 1246494 & 166 & 19 & 30 \\
173 & 26673 & 26673 & 2293878 & 86 & 54 & 62 \\
179 & 32388 & 2699 & 480422 & 178 & 60 & 132 \\
181 & 35447 & 35447 & 638046 & 18 & 138 & 149 \\
191 & 38384 & 19192 & 3646480 & 190 & 28 & 163 \\
193 & 37210 & 37210 & 7144320 & 192 & 114 & 120 \\
197 & 34920 & 4365 & 427770 & 98 & 61 & 22 \\
199 & 41888 & 10472 & 1036728 & 99 & 65 & 180 \\
211 & 45849 & 15283 & 229245 & 15 & 134 & 137 \\
223 & 49121 & 49121 & 5452431 & 111 & 9 & 126 \\
227 & 56510 & 28255 & 6385630 & 226 & 33 & 162 \\
229 & 54829 & 54829 & 6250506 & 114 & 3 & 62 \\
233 & 53520 & 4460 & 1034720 & 232 & 212 & 207 \\
239 & 56584 & 7073 & 1683374 & 238 & 202 & 207 \\
241 & 66112 & 33056 & 793344 & 24 & 32 & 226 \\
251 & 64724 & 32362 & 1618100 & 50 & 226 & 204 \\
257 & 63176 & 31588 & 4043264 & 128 & 143 & 165 \\
263 & 70608 & 35304 & 9249648 & 262 & 258 & 189 \\
269 & 71024 & 8878 & 1189652 & 134 & 170 & 24 \\
271 & 73020 & 4868 & 262872 & 54 & 266 & 188 \\
277 & 74418 & 24806 & 6846456 & 276 & 24 & 115 \\
281 & 80956 & 80956 & 22667680 & 280 & 259 & 267 \\
283 & 80436 & 6703 & 1890246 & 282 & 81 & 272 \\
293 & 84592 & 21148 & 3087608 & 146 & 172 & 267 \\
307 & 94816 & 47408 & 4835616 & 102 & 155 & 51 \\
311 & 105052 & 52526 & 16283060 & 310 & 289 & 124 \\
313 & 97720 & 24430 & 635180 & 26 & 255 & 265 \\
317 & 108842 & 108842 & 34394072 & 316 & 126 & 115 \\
331 & 102800 & 25700 & 1413500 & 55 & 172 & 274 \\
337 & 116852 & 29213 & 2453892 & 84 & 196 & 147 \\
347 & 125596 & 31399 & 10864054 & 346 & 38 & 280 \\
349 & 113967 & 5427 & 314766 & 58 & 110 & 115 \\
353 & 125906 & 62953 & 5539864 & 88 & 336 & 317 \\
359 & 129600 & 64800 & 23198400 & 358 & 105 & 254 \\
367 & 136161 & 45387 & 16611642 & 366 & 268 & 360 \\
373 & 146336 & 4573 & 283526 & 62 & 31 & 97 \\
379 & 143613 & 143613 & 54285714 & 378 & 189 & 293 \\
383 & 153214 & 76607 & 29263874 & 382 & 64 & 157 \\
389 & 160166 & 80083 & 15536102 & 194 & 311 & 355 \\
397 & 165192 & 6883 & 1362834 & 198 & 121 & 119 \\
\hline
\end{longtable}

\begin{rem}
Among the primes $p \leq 400$, for $p \neq 2, 3, 5, 7, 29, 41, 47, 379$, we have
    \[
		\ord_p(D_P) \mid \Per_p(\seq{c}) \mid (p-1) \ord_p(D_P)
    \]
by Theorem \ref{MainTheorem}.
For the excluded primes, the curve $C$ has bad reduction at
$p = 2, 5, 29$.
For $p = 7, 41, 47, 379$,
although we cannot apply Theorem \ref{MainTheorem} because
$p$ divides $c_3 c_4 c_5 c_6 c_7 (c_4^3 - c_3^3 c_5)$,
we observe that the above divisibilities hold for such $p$.
However, for $p=3$, we observe that the divisibility
$\ord_p(D_P) \mid \Per_p(\seq{c})$
holds, but the divisibility
$\Per_p(\seq{c}) \mid (p-1) \ord_p(D_P)$
does not.
\end{rem}

\begin{rem}\label{rem:d-statistic}
For primes $\leq 400$,
we have $\Per_p(\seq{c}) = \ord_p(D_P)$
for $p = 31$ only.
We have $\Per_p(\seq{c}) = (p-1) \ord_p(D_P)$
for $p = 17$, $23$, $61$, $67$, $89$, $97$, $107$, $113$, $137$, $139$, $151$, $163$, $167$, $179$, $191$, $193$, $227$, $233$, $239$, $263$, $277$, $281$, $283$, $311$, $317$, $347$, $359$, $367$, $379$, $383$.
\end{rem}

\subsection*{Acknowledgements}

The authors would like to thank the referee for useful comments and suggestions.
The work of Y.\ I.\ was supported by JSPS KAKENHI Grant Number 21K18577, 24K21512 and 21K13773.
The work of T.\ I.\ was supported by JSPS KAKENHI Grant Number 21K18577, 24K21512 and 23K20786.
The work of T.\ O.\ was supported by JSPS KAKENHI Grant Number 18H05233, 20K14295, 21K18577 and 24K21512.
The work of T.\ T.\ was supported by JSPS KAKENHI Grant Number 21K18577, 24K21512 and 22H01115.
The work of Y.\ U.\ was supported by JSPS KAKENHI Grant Number 21K18577, 24K21512 and 20K03517.
A part of this work was done while the authors were supported by the Sumitomo Foundation FY2018 Grant for Basic Science Research Projects (Grant Number 180044).
Most of calculations were done with the aid of  the computer algebra systems Magma \cite{Magma} and Sage \cite{Sage}.

\end{document}